 \documentclass[tablecaption=bottom,wcp]{jmlr} % W&CP article

\usepackage{booktabs}
\usepackage[load-configurations=version-1]{siunitx} % newer version
\theorembodyfont{\upshape}
\theoremheaderfont{\scshape}
\theorempostheader{:}
\theoremsep{\newline}

\DeclareMathOperator*{\argmin}{arg\,min}
\DeclareMathOperator*{\argmax}{arg\,max}

\usepackage{amsmath}

\numberwithin{equation}{section}

\jmlrvolume{96}
\jmlryear{2018.}
\jmlrsubmitted{02/15/2019}
\jmlrpublished{publication date}
\jmlrworkshop{1st Symposium on Advances in Approximate Bayesian Inference} % W&CP title
\jmlrproceedings{AABI 2018}{Proceedings of Machine Learning Research}

\title[Consistency of ELBO maximization]{Consistency of ELBO maximization for model selection}

  \author{\Name{Badr-Eddine Ch\'erief-Abdellatif} \Email{badr.eddine.cherief.abdellatif@ensae.fr} \\ 
   \addr CREST, ENSAE, Universit\'e Paris Saclay}

\begin{document}

\maketitle

\begin{abstract}
The Evidence Lower Bound (ELBO) is a quantity that plays a key role in variational inference. It can also be used as a criterion in model selection. However, though extremely popular in practice in the variational Bayes community, there has never been a general theoretic justification for selecting based on the ELBO. In this paper, we show that the ELBO maximization strategy has strong theoretical guarantees, and is robust to model misspecification while most works rely on the assumption that one model is correctly specified. We illustrate our theoretical results by an application to the selection of the number of principal components in probabilistic PCA.
\end{abstract}

\section{Introduction}
\label{sec:intro}

Approximate Bayesian inference is at the core of modern Bayesian statistics and machine learning. While exact Bayesian inference is often intractable, variational inference has proved to provide an efficient solution when dealing with large datasets and complex probabilistic models. Variational Bayes (VB) aims at maximizing a numerical quantity referred to as Evidence Lower Bound on the marginal likelihood (ELBO), and thus makes use of optimization techniques to converge faster than Monte Carlo sampling approach. ~\cite{blei2017variational} provides a comprehensive survey on variational inference. Although VB is mainly used for its practical efficiency, little attention has been put towards its theoretical properties during the last years. While ~\cite{alquier2016properties} studied the properties of variational approximations of Gibbs distributions used in machine learning for bounded loss functions, \cite{Tempered,Chicago,wang2018frequentist,Plage,cherief2018consistency} extended the results to more general statistical models. 

At the same time, model selection remains a major problem of interest in statistics that naturally arises in the course of scientific inquiry. The statistician aims at selecting a model among several candidates given an observed dataset. To do so, one can perform cross validation as in \cite{CVVehtari} or maximize a numerical criterion to make the final choice, see the review of \cite{ModelSelectionReview}. In the literature, penalized criteria such as AIC and BIC respectively introduced by \cite{AIC} and \cite{BIC} are popular. While AIC aims at optimizing the prediction performance, BIC is more suitable for recovering with high probability the true model (when such a model exists), see \cite{yang2005can}. Thus, it is necessary to define a criterion suited to a given objective. Meanwhile, a non-asymptotic theory of penalization using oracle inequalities has been developed during the last two decades, and offers a simple way to assess the quality of a given model selection criterion. We refer the interested reader to \cite{MR2319879} for more details.

In this paper, we are interested in finding an estimate of the distribution of the data, and we need to
choose from among competing models. ~\cite{blei2017variational} states that "the [evidence lower] bound is a good approximation of the marginal likelihood, which provides a basis for selecting a model. Though this sometimes works in practice, selecting based on a bound is not justified in theory". Since then, authors of~\cite{cherief2018consistency} have provided an analysis of model selection based on the ELBO in the case of mixture models. We extend their result to the general case of independent and identically distributed (i.i.d.) data, and we provide an oracle inequality on the ELBO criterion that justifies the consistency of ELBO maximization when the objective is the estimation of the distribution of the data. In particular, as soon as there exists a true model, we show that the ELBO criterion is adaptive and that the selected estimator achieves the same convergence rate than the variational approximation associated with the true model.

The rest of this paper is organized as follows. Section \ref{sec:notations} introduces the setting and the key concepts needed to understand our results. In Section \ref{sec:consistency}, we prove that the ELBO criterion provides a variational approximation that is consistent with the sample size as soon as there exists a true model. We also extend the result to misspecified models. We finally illustrate the main theorem of this paper by an application to the selection of the number of principal components in probabilistic Principal Component Analysis (PCA) in Section \ref{sec:example}. All the proofs are deferred to the appendix.

\section{Framework}
\label{sec:notations}

Let us introduce the notations and the framework we adopt in this paper. We consider a collection of i.i.d. random variables $X_1$,...,$X_n$ distributed according to some probability distribution $P^0$ in a measurable space $\big( \mathbb{X},\mathcal{X} \big)$. We denote $X_1^n=(X_1,...,X_n)$. We consider a countable collection $ \left\lbrace \mathcal{M}_K /  K \geq 1 \right\rbrace$  of statistical mixture models $\mathcal{M}_K = \lbrace P_{\theta_K} \hspace{0.1cm} / \hspace{0.1cm} \theta_{K} \in \Theta_K \rbrace$ where $\Theta_K$ is the parameter set associated with index $K$. We make no assumptions on $\Theta_K$'s nor on $P_{\theta_K}$. Parameter spaces may overlap or have inclusion relationships. Let $\mathcal{M}_{1}^{+}(\Theta_{K})$ be the set of all probability distributions over $\Theta_{K}$.

We use a Bayesian approach, and we define a prior $\pi$ over the full parameter space $\cup_{K \geq 1} \Theta_K$ (equipped with some suited sigma-algebra). First, we specify a prior weight $\pi_K$ assigned to model $\mathcal{M}_K$, and then a conditional prior $\Pi_K(.)$ on $\theta_K \in \Theta_K$ given model $\mathcal{M}_K$: 
$$ \pi = \sum_{K \geq 1} \pi_K \Pi_K. $$

The Kullback-Leibler divergence between two probability distributions $P$ and $R$ is
$$
\textrm{KL}(P,R) = \begin{cases}
\int \log\left( \frac{dP}{dR} \right) dP \hspace{0.2cm} \text{if $R$ dominates $P$}, \\
+ \infty \hspace{0.2cm} \text{otherwise.}
\end{cases}
$$
For any $\alpha \ne 1$, authors of \cite{van2014renyi} detail the properties of the $\alpha$-Renyi divergence between two probability distributions $P$ and $R$
which is equal to:
$$
D_\alpha(P,R) = \begin{cases}
\frac{1}{\alpha-1} \log \int \left( \frac{dP}{dR} \right) ^{\alpha-1} dP \hspace{0.2cm} \text{if $R$ dominates $P$}, \\
+ \infty \hspace{0.2cm} \text{otherwise.}
\end{cases}
$$

We define the tempered posterior distribution $\pi_{n,\alpha}^K(.|X_1^n)$ on parameter $\theta_K \in \Theta_K$ given model $\mathcal{M}_K$ using prior $\Pi_K$ and likelihood $L_n$ for any $\alpha \in (0,1)$:
\[
\pi_{n,\alpha}^K(d\theta_K|X_1^n) \propto L_{n}(\theta_K)^\alpha \Pi_K(d\theta_K).
\]
This definition is a slight variant of the regular Bayesian posterior (for which $\alpha=1$), and is also referred to as Bayesian fractional posterior in \cite{bhattacharya2016bayesian}. This posterior is easier to sample from, more robust to model misspecification and requires less stringent conditions to obtain consistency, see respectively \cite{behrens2012tuning}, \cite{grunwaldmisspecifiation} and \cite{bhattacharya2016bayesian}.

The Variational Bayes approximation $\tilde{\pi}_{n,\alpha}^K(.|X_1^n)$ of the tempered posterior associated with model $\mathcal{M}_K$ is then defined as the projection, with respect to the Kullback-Leibler divergence, of the tempered posterior onto some set $\mathcal{F}_K$:
$$
\tilde{\pi}_{n,\alpha}^K(.|X_1^n) = \argmin_{\rho_K \in \mathcal{F}_K}  \textnormal{KL}(\rho_K,\pi_{n,\alpha}^K(.|X_1^n)).
$$

The choice of the variational set $\mathcal{F}_K$ is crucial: the variational approximation must be close enough to the target distribution (as an approximation of the tempered posterior) but not too close (in order to be tractable). A classical variational set $\mathcal{F}_K$ is the parametric family which leads to a tractable parametric approximation, e.g. a Gaussian distribution. Another popular set $\mathcal{F}_K$ in the VB community is the mean-field approximation that is based on a partition of the space of parameters, and which consists in a factorization of the variational approximation over the partition.

Alternatively, the variational approximation is often defined as the distribution into $\mathcal{F}_K$ that maximizes the Evidence Lower Bound:
\[
\tilde{\pi}_{n,\alpha}^K(.|X_1^n) = \argmax_{\rho_K \in \mathcal{F}_K} \bigg\{ \alpha \int \ell_n(\theta_K) \rho_K(d\theta_K) - \textrm{KL}\big(\rho_K,\Pi_K\big) \bigg\}
\]
where the function inside the argmax operator is the ELBO (as a function of $K$ and $\rho_K$) and $\ell_n$ is the log-likelihood. In the following, we will just call $\textrm{ELBO}(K)$ the closest approximation to the log-evidence, i.e. the value of the ELBO evaluated at its maximum:
$$
\textrm{ELBO}(K)=\alpha \int \ell_n({\theta}_K) \tilde{\pi}_{n,\alpha}^K(d\theta_K|X_1^n)-\textrm{KL}(\tilde{\pi}_{n,\alpha}^K(.|X_1^n),\Pi_K).
$$

In the variational Bayes community, researchers and practitioners use the ELBO in order to select the model from which they will consider the final variational approximation $\tilde{\pi}_{n,\alpha}^{\hat{K}}(.|X_1^n)$, as stated in \cite{blei2017variational}. We propose to consider a penalized version of the ELBO criterion
$$
\hat{K} = \argmax_{K\geq 1} \bigg\{ \textrm{ELBO}(K) - {\log\bigg(\frac{1}{\pi_K}\bigg)} \bigg\}
$$
which is a slight variant of the classical definition, although choosing a uniform prior over a finite number of models leads to maximizing the ELBO. Note that the penalty term is not just an artefact in order to ease the theoretical proof, but it is a complexity term that reflects our prior beliefs over the different models.

We will provide in the next section a theoretical justification to such a selection criterion and show that the selected variational estimator $\tilde{\pi}_{n,\alpha}^{\hat{K}}(.|X_1^n)$ is consistent under mild conditions as soon as there exists a true model. We will adopt the definition of \textit{consistency} used in \cite{Tempered} and \cite{cherief2018consistency} that is, the Bayesian estimator is said to be consistent if, in expectation (with respect to the random variables distributed according to $P^0$), the average Renyi loss between a distribution in the selected model and the true distribution (over the Bayesian estimator) goes to zero as $n\rightarrow +\infty$:
$$
\mathbb{E} \bigg[ \int D_{\alpha}( P_{\theta}, P^0 ) \tilde{\pi}_{n,\alpha}^{\hat{K}}(d\theta|X_1^n) \bigg] \xrightarrow[n\rightarrow +\infty]{} 0.
$$
This definition is closely related to the notion of \textit{concentration} which is defined in \cite{ghosal2000convergence} as the asymptotic concentration of the Bayesian estimator around the true distribution, and which is usually used to assess frequentist guarantees for Bayesian estimators. It is sometimes also referred to as \textit{contraction} (or even \textit{consistency}). See Appendix \ref{apd:connection} for more details on the connection between the notions of \textit{consistency} and \textit{concentration}.

\section{Consistency of the ELBO criterion}
\label{sec:consistency}

In this section, unless explicitly stated otherwise, we assume that there exists a true model $\mathcal{M}_{K_0}$ that contains the true distribution $P^0$, i.e. that there exists $K_0$ and $\theta^0 \in \Theta_{K_0}$ such that $P^0=P_{\theta^0}$. 

A key assumption introduced in \cite{ghosal2000convergence} in order to obtain the concentration of the regular posterior distribution ${\pi}_{n,1}^{{K}_0}(.|X_1^n)$ associated with the true model $\mathcal{M}_{K_0}$ is a \textit{prior mass condition} which states that the prior $\Pi_{K_0}$ must give enough mass to some neighborhood (in the Kullback-Leibler sense) of the true parameter. \cite{bhattacharya2016bayesian} showed that this condition was sufficient when considering tempered posteriors ${\pi}_{n,\alpha}^{{K}_0}(.|X_1^n)$. \cite{Tempered} extended this assumption in order to obtain the concentration and the consistency of variational approximations of the tempered posteriors $\tilde{\pi}_{n,\alpha}^{{K}_0}(.|X_1^n)$. In addition to the previous prior mass condition, this extension requires the variational set $\mathcal{F}_{K_0}$ to contain probability distributions concentrated around the true parameter. Note that when $\mathcal{F}_{K_0} = \mathcal{M}_{1}^{+}(\Theta_{K_0})$, this goes back to the standard prior mass condition. This extended prior mass condition is standard in the variational Bayes community, see \cite{Tempered,cherief2018consistency}, and can be formulated as follows:

\vspace{0.2cm}

\noindent
\textbf{\textit{Assumption :}}
\textit{We assume that there exists $r_n$ for which there is a distribution $\rho_{K_0,n} \in \mathcal{F}_{K_0}$ such that}:
\begin{equation}
\label{priormass}
  \int \textrm{KL}(P^0,P_{\theta_{K_0}}) \rho_{K_0,n}(d\theta_{K_0}) \leq r_{n} \hspace{0.2cm} \textnormal{and} \hspace{0.2cm} 
  \textrm{KL}(\rho_{K_0,n},\Pi_{K_0}) \leq n r_{n}.
\end{equation}

\begin{remark}
\label{rmk-1} 
Define the KL-ball $\mathcal{B}$ centered at $\theta_0$ of radius $r_n$:
$$ \mathcal{B} = \{ \theta \in \Theta_{K_0} / \hspace{0.1cm} \textnormal{KL}(P_{\theta_0},P_\theta) \leq r_n \} , $$
and consider the restriction $\rho_{K_0,n}$ of $\Pi_{K_0}$ to $\mathcal{B}$. Then it is clear that when  $\rho_{K_0,n} \in \mathcal{F}_{K_0}$, Assumption \ref{priormass} becomes equivalent to the former prior mass condition of \cite{ghosal2000convergence}, i.e. $\Pi_{K_0}(\mathcal{B}) \geq e^{-nr_n}$. The computation of the prior mass $\Pi_{K_0}(\mathcal{B})$ is a major difficulty. It has been raised as a question of interest in \cite{ghosal2000convergence}, and is addressed for categorical distributions and Dirichlet priors in \cite{ghosal2000convergence} (but for an $L_1$-ball) and in \cite{cherief2018consistency} (for a KL-ball). Unfortunately, $\rho_{K_0,n}$ does not belong to $\mathcal{F}_{K_0}$ in general and the computation of the prior mass is no longer sufficient. Nevertheless, the strategy of computing the prior mass of KL-balls remains of interest when dealing with mixture models and mean-field approximation sets, see \cite{cherief2018consistency} where the authors showed that studying the prior mass condition of \cite{ghosal2000convergence} independently on the weights and on each component becomes sufficient.
\end{remark}

\begin{remark}
\label{rmk-2} 
When $\mathcal{F}_{K_0}$ is parametric, it is often possible to overcome the difficulty presented above in order to find a rate $r_n$ as in Assumption \ref{priormass}. Indeed, the point is to express the distribution $\rho_{K_0,n}$ using the general parametric form of the variational family, and to find relevant values of the parameters that will lead to fast rates of convergence $r_n$. This is the strategy we follow in Section \ref{sec:example} for probabilistic PCA. See \cite{Tempered,cherief2018consistency} for other examples of such computations.

\end{remark}

\cite{Tempered} showed that the variational approximation $\tilde{\pi}_{n,\alpha}^{{K}_0}(.|X_1^n)$ associated with a true model is consistent under Assumption \ref{priormass} and that the convergence rate is equal to $r_n$. Nevertheless, in model selection, we do not necessarily know which model is true and the challenge is to be able to find one such that the corresponding approximation is consistent at a comparable convergence rate. We show that the variational approximation $\tilde{\pi}_{n,\alpha}^{\hat{K}}(.|X_1^n)$ associated with the selected model is also consistent at rate $r_n$ as soon as Assumption \ref{priormass} is satisfied:

\begin{theorem}
\label{thm-true-model}
Assume that Assumption \ref{priormass} is satisfied. Then for any $\alpha \in (0,1)$,
\begin{equation*}
\mathbb{E} \bigg[ \int D_{\alpha}( P_{\theta}, P^0 ) \tilde{\pi}_{n,\alpha}^{\hat{K}}(d\theta|X_1^n) \bigg] \leq \frac{1+\alpha}{1-\alpha} r_{n} +  \frac{\log(\frac{1}{\pi_{K_0}})}{n(1-\alpha)}.
\end{equation*}
\end{theorem}

The inequality in Theorem \ref{thm-true-model} shows the adaptivity of our procedure. Indeed, whatever the value of $\hat{K}$ (which can be different from $K_0$), we obtain the consistency of the selected variational approximation at the same rate of convergence than the estimator associated with the true model (as soon as the additional term in the upper bound is lower than $r_n$, which is the case for prior weights used in practice). We recall that we look for a good estimation of the true distribution $P^0$ and not for an estimation of the true model index $K_0$ which is a different task that would require identifiability assumptions that are stronger than those in our theorem. The overall rate is composed of the convergence rate associated with the true model $\mathcal{M}_{K_0}$, and of a complexity term that reflects the prior belief over the (unknown) true model. For example, if we range a countable number of models according to our prior belief, and we take $\pi_K=2^{-K}$, then the corresponding term will be of order $K_0/n$. More generally, when $\frac{1}{n}\lesssim r_{n}$, we obtain the consistency at the rate associated with the true model.

As a short example, \cite{cherief2018consistency} investigated the case of mixture models. For instance, authors obtained a convergence rate equal to ${K_0 \log(nK_0)}/{n}$ for Gaussian mixtures when there exists a true $K_0$-components mixture model. We study another example in Section \ref{sec:example}.

We can also extend this result to misspecified models. In the model selection literature, only little attention has been put to misspecification when the true distribution does not belong to any of the models, see \cite{modelselectionmisspecification}. Now, we do not assume any longer that there exists a true model, and we show that our ELBO criterion is robust to model misspecification:

\begin{theorem}
\label{thm-extension}
For each index $K$, let us define the set $\Theta_K(r_{K,n})$ of parameters $\theta^*_K \in \Theta_{K}$, for which there is a distribution $\rho_{K,n} \in \mathcal{F}_{K}$ such that:
\begin{equation}
\label{cond}
  \int \mathbb{E}\bigg[\log \frac{P_{\theta^*_K}(X_i)}{P_{\theta_{K}}(X_i)}\bigg] \rho_{K,n}(d\theta_{K}) \leq r_{K,n} \hspace{0.2cm} \textnormal{and} \hspace{0.2cm} 
  \textnormal{KL}(\rho_{K,n},\Pi_{K}) \leq n r_{K,n}.
\end{equation}
Then for any $\alpha \in (0,1)$,
\begin{multline*}
\mathbb{E} \bigg[ \int D_{\alpha}( P_{\theta}, P^0 ) \tilde{\pi}_{n,\alpha}^{\hat{K}}(d\theta|X_1^n) \bigg]  \\ \leq \inf_{K \geq 1} \bigg\{ \frac{\alpha}{1-\alpha} \inf_{\theta^*_K \in \Theta_K(r_{K,n})} \textnormal{KL}(P^0,P_{\theta^*_K}) + \frac{1+\alpha}{1-\alpha} r_{K,n} +  \frac{\log(\frac{1}{\pi_{K}})}{n(1-\alpha)} \bigg\}.
\end{multline*}
\end{theorem}

Note that when there exists a true model $\mathcal{M}_{K_0}$ such that $P^0=P_{\theta^0}$ with $\theta^0 \in \Theta_{K_0}$, then under Assumption \ref{priormass},  we get $\theta^0 \in\Theta_{K_0}(r_{K_0,n}) $, and we recover Theorem \ref{thm-true-model}. Furthermore, the oracle inequality in Theorem \ref{thm-extension} shows that the selected variational approximation adaptively achieves the best upper bound among the different models $\mathcal{M}_K$, where each upper bound is a trade-off between two terms: a bias due to the error of approximating the true distribution by a distribution in model $\mathcal{M}_K$, and a variance term $r_{K,n}$ (as soon as the penalty term is lower than $r_{K,n}$) that is defined in Condition \ref{cond}.

\section{Application to probabilistic PCA}
\label{sec:example}

We consider here the probabilistic Principal Component Analysis (PCA) problem as an application of our work. From now on, matrices will be denoted in bold capital letters. We assume the model 
$$X_i=\textit{\textbf{W}} Z_i + \sigma^2 \textit{\textbf{I}}_d$$
with i.i.d. Gaussian random variables $Z_i \sim \mathcal{N}(0,\textit{\textbf{I}}_K)$, where $\textit{\textbf{I}}_d$ and $\textit{\textbf{I}}_K$ are respectively the $d$- and $K$-dimensional identity matrices ($K<d$), $\textit{\textbf{W}} \in \mathbb{R}^{d\times K}$ is the $K$-rank matrix that contains the principal axes and $\sigma^2$ is a noisy term that is known. We suppose here that data are centred. Hence, the distribution of each $X_i$ is
$$P_\textit{\textbf{W}}:=\mathcal{N}(0,\textit{\textbf{W}}\textit{\textbf{W}}^T + \sigma^2 \textit{\textbf{I}}_d).$$ 
We are not interested here in estimating the principal axes $\textit{\textbf{W}}$ and selecting the number of components $K$, but in estimating the true distribution of the $X_i$'s.

Each model corresponds to a rank $K$. We place an equal prior weight over each integer $K=1,...,d$. Hence the optimization problem is equivalent to maximizing the ELBO as in \cite{blei2017variational}. Given rank $K$, we place a prior over the $K$-rank matrix $\textit{\textbf{W}}$ to infer a distribution over principal axes. We choose independent Gaussian priors $\mathcal{N}(0,s^2\textit{\textbf{I}}_d)$ on the columns $W_1,...,W_K$ of $\textit{\textbf{W}}$. We also consider Gaussian independent variational approximations $\mathcal{N}(\mu_j,{\bold{\Sigma}}_j)$ for the columns of $\textit{\textbf{W}}$. Then, as soon as there exists a true model, i.e. there exists $K_0$ and $\textit{\textbf{W}}_0 \in \mathbb{R}^{d\times K_0}$ such that the true distribution of each $X_i$ is $P_{\textit{\textbf{W}}_0}=\mathcal{N}(0,\textit{\textbf{W}}_0\textit{\textbf{W}}_0^T + \sigma^2 \textit{\textbf{I}}_d)$, under the assumption that the coefficients of $\textit{\textbf{W}}_0$ are bounded, then Theorem \ref{corPCA} provides an explicit rate of convergence of our variational estimator even when $K_0$ is unknown:

\begin{theorem}
\label{corPCA}
For any $\alpha \in (0,1)$, as soon as there exists a true model $\mathcal{M}_{K_0}$ such that $P^0=P_{\textit{\textbf{W}}_0}$ with $\textit{\textbf{W}}_0 \in \mathbb{R}^{d\times K_0}$ and such that the coefficients of $\textit{\textbf{W}}_0$ are bounded, then:
\begin{equation*}
\mathbb{E} \bigg[ \int D_{\alpha}( P_{\textit{\textbf{W}}}, P_{\textit{\textbf{W}}_0} ) \tilde{\pi}_{n,\alpha}^{\hat{K}}(d\textit{\textbf{W}}|X_1^n) \bigg] = \mathcal{O}\left( \frac{dK_0 \log(dn) }{n} \right) .
\end{equation*}
\end{theorem}

The proof as well as the computation of the ELBO are detailed in the appendix. Note that this corollary can directly lead to a result in Frobenius distance between covariance matrices $\textit{\textbf{W}}\textit{\textbf{W}}^T+\sigma^2 \textit{\textbf{I}}_d$ and $\textit{\textbf{W}}_0\textit{\textbf{W}}_0^T+\sigma^2 \textit{\textbf{I}}_d$ instead of the Renyi divergence between the corresponding distributions even when $\textit{\textbf{W}}$ and $\textit{\textbf{W}}_0$ are not equal-sized matrices. We denote $\|.\|_F$ the Frobenius norm and $\|.\|_2$ the spectral norm of a matrix, which are respectively defined as the square root of the sum of the absolute squares of the elements of a matrix and as its largest singular value.

The following corollary assesses the consistency of the selected variational approximation to the true covariance matrix in Frobenius norm.  The idea, borrowed from \cite{Tempered}, is to project matrices onto some set of bounded matrices under the assumption that the spectral norm of the true matrix $\textit{\textbf{W}}_0$ is also bounded:

\begin{corollary}
\label{corPCAclip}
For any $\alpha \in (0,1)$, as soon as there exists a true model $\mathcal{M}_{K_0}$ such that $P^0=P_{\textit{\textbf{W}}_0}$ with $\textit{\textbf{W}}_0 \in \mathbb{R}^{d\times K_0}$ and such that the spectral norm of $\textit{\textbf{W}}_0$ is upper bounded by a positive constant $B>0$, then:
\begin{equation*}
\mathbb{E} \bigg[ \int \big\| \textrm{clip}_B(\textit{\textbf{W}}\textit{\textbf{W}}^T) - \textit{\textbf{W}}_0 \textit{\textbf{W}}_0^T \big\|_F^2 \tilde{\pi}_{n,\alpha}^{\hat{K}}(d\textit{\textbf{W}}|X_1^n) \bigg] = \mathcal{O}\left( \frac{dK_0 \log(dn) }{n} \right)
\end{equation*}
where $\textrm{clip}_B(\textit{\textbf{A}})$ is the matrix which $(i,j)$-entry is equal to $ \left\{
\begin{array}{l}
  \textit{\textbf{A}}_{i,j} \hspace{0.3cm} \textrm{if} \hspace{0.2cm} |\textit{\textbf{A}}_{i,j}| \leq B^2 \\
  B^2 \hspace{0.5cm} \textrm{if} \hspace{0.2cm} \textit{\textbf{A}}_{i,j} \geq B^2 \\ - B^2 \hspace{0.2cm} \textrm{otherwise}. \hspace{0.2cm}
\end{array}
\right.$
\end{corollary}

The requirement in our corollary is that the spectral norm of the true matrix $\textit{\textbf{W}}_0$ is bounded by some positive constant $B$, which implies the boundedness of the coefficients of the matrix as required in Theorem \ref{corPCA}. In particular, the coefficients of the matrix $\textit{\textbf{W}}_0 \textit{\textbf{W}}_0^T$ are bounded by $B^2$:
\begin{align*}
|(\textit{\textbf{W}}_0 \textit{\textbf{W}}_0^T)_{i,j}| = \bigg| \sum_{k=1}^{K_0} (\textit{\textbf{W}}_0)_{i,k} (\textit{\textbf{W}}_0)_{j,k} \bigg| & \leq \bigg( \sum_{k=1}^{K_0} (\textit{\textbf{W}}_0)_{i,k}^2 \bigg)^{1/2} \bigg( \sum_{k=1}^{K_0} (\textit{\textbf{W}}_0)_{j,k}^2 \bigg)^{1/2} \\
& = \frac{\| \textit{\textbf{W}}_0 e_i \|_2}{\|e_i\|_2} \frac{\| \textit{\textbf{W}}_0 e_j \|_2}{\|e_j\|_2} \leq \| \textit{\textbf{W}}_0 \|_2^2 \leq B^2 
\end{align*}
using Cauchy-Schwarz inequality and the property $\|\textit{\textbf{W}}_0\|_2 = \max_{x\ne 0} \frac{\|\textit{\textbf{W}}_0 x \|_2}{\|x\|_2}$ where $e_\ell$ is the vector of $\mathbb{R}^d$ which components are all equal to $0$ except for the $\ell$-th one that is set to $1$.
Hence it seems sensible to project (with respect to the Frobenius distance) any estimator $\textit{\textbf{W}} \textit{\textbf{W}}^T$ onto the set of all matrices whose entries lie in the interval $[- B^2, B^2]$, which is exactly what the $\textit{clip}_B$ application does. Note that the spectral norm of Matrix $\textit{\textbf{W}}_0$ is equal to the largest eigenvalue of $\textit{\textbf{W}}_0 \textit{\textbf{W}}_0^T$, so our assumption comes back to upper bounding the eigenvalues of the covariance matrix $\textit{\textbf{W}}_0 \textit{\textbf{W}}_0^T+\sigma^2 \textit{\textbf{I}}_d$, which is a classical assumption when estimating covariance matrices, see for instance \cite{MatrixEstimationWu}.

It is also possible to obtain a consistent pointwise covariance matrix estimator with the same convergence rate:

\begin{corollary}
\label{corPCAclipfreq}
For any $\alpha \in (0,1)$, as soon as there exists a true model $\mathcal{M}_{K_0}$ such that $P^0=P_{\textit{\textbf{W}}_0}$ with $\textit{\textbf{W}}_0 \in \mathbb{R}^{d\times K_0}$ and such that the spectral norm of $\textit{\textbf{W}}_0$ is bounded by $B$. Let us define a pointwise estimator of the covariance matrix:
$$
\hat{\mathit{\bold{\Sigma}}} = \int \textrm{clip}_B(\textit{\textbf{W}}\textit{\textbf{W}}^T) \tilde{\pi}_{n,\alpha}^{\hat{K}}(d\textit{\textbf{W}}|X_1^n) + \sigma^2 \textit{\textbf{I}}_d.
$$
Then,
\begin{equation*}
\mathbb{E} \bigg[ \big\| \hat{\bold{\Sigma}} - (\textit{\textbf{W}}_0 \textit{\textbf{W}}_0^T+\sigma^2 \textit{\textbf{I}}_d) \big\|_F^2 \bigg] = \mathcal{O}\left( \frac{dK_0 \log(dn) }{n} \right).
\end{equation*}
\end{corollary}

\section*{Discussion}

In this paper we proved the consistency of ELBO maximization in model selection. By penalizing the variational lower bound using our prior beliefs over the different models, we showed that under mild conditions, the variational approximation associated with the selected model is consistent at the same convergence rate than the approximation associated with the true model. Moreover, the oracle inequality in Theorem \ref{thm-extension} proved that the selected approximation is robust to misspecification. An application to the selection of the number of principal components in probabilistic PCA was provided as a short example.

We discuss in Appendix \ref{apd:connection} the connection between the notions of \textit{consistency} and \textit{concentration}. This justifies the use of the $\alpha$ parameter in the definition of the evidence lower bound, as the regular posterior distribution is not robust to model misspecification. Indeed, authors of \cite{grunwaldmisspecifiation} explain that there are pathologic cases where the regular posterior does not concentrate to the true distribution.

A point of interest when dealing with model selection is the question of recovering the true model (when it exists). This issue falls beyond the scope of this paper which treats the question of estimating the true distribution, and can be the object of future works. The true model recovery would require stronger assumptions, but the implementation in Section 5 in \cite{VariationalComponents} suggests that those may hold for probabilistic PCA.

Also, it would be interesting to study cross-validation instead of ELBO maximization. However, the tools used in this work such as the theory of penalized criteria and oracle inequalities were particularly suited to the ELBO, and thus a different theory should be used in order to obtain the consistency of validation log-likelihood in the VB framework. This question is left for future research.

\acks{I would like to warmly thank Pierre Alquier, Lionel Riou-Durand and the anonymous referees for their inspiring comments and suggestions on this work.}

\bibliographystyle{apalike}

\newpage

\appendix

\section{Connection between consistency and concentration.}\label{apd:connection}

In this appendix, we highlight the connection between the notions of \textit{consistency} used in \cite{Tempered} and \cite{cherief2018consistency} and \textit{concentration}. We consider a true model $\mathcal{M}_{K_0}$ to which the true distribution $P^0=P_{\theta^0}$ belongs, $\theta^0 \in \Theta_{K_0}$. We recall that the Bayesian estimator $ \tilde{\pi}_{n,\alpha}^{\hat{K}}(.|X_1^n) $ is said to be consistent if, in expectation (with respect to the random variables distributed according to $P^0$), the average Renyi loss between a distribution in the selected model and the true distribution (over the Bayesian estimator) goes to zero as $n\rightarrow +\infty$:
$$
\mathbb{E} \bigg[ \int D_{\alpha}( P_{\theta}, P^0 ) \tilde{\pi}_{n,\alpha}^{\hat{K}}(d\theta|X_1^n) \bigg] \xrightarrow[n\rightarrow +\infty]{} 0.
$$
Similarly, we give the definition of \textit{concentration} at rate $s_n$ of the selected variational approximation to $P^0$ as stated in \cite{ghosal2000convergence}, that is, in probability (with respect to the random variables distributed according to $P^0$), the approximation concentrates asymptotically around the true distribution as $n\rightarrow +\infty$, i.e. in probability:
$$
\tilde{\pi}_{n,\alpha}^{\hat{K}}\bigg(D_{\alpha}( P_{\theta}, P^0 )>M s_n|X_1^n\bigg) \xrightarrow[n\rightarrow +\infty]{} 0
$$
for any constant $M>0$. The reference metric here is the $\alpha$-Renyi divergence.

We show in this appendix that the consistency of the selected variational approximation to $P^0$ at rate $r_n$ implies the concentration of the selected variational approximation to $P^0$ at any rate $s_n$ such that $r_n=o(s_n)$ and $s_n \rightarrow 0$ as $n\rightarrow +\infty$, as for instance $s_n=r_n \log(\log(n))$ when the consistency rate $r_n$ is slower than a log-logarithmic one. 

To do so, we assume that the selected variational approximation is consistent to $P^0$ at rate $r_n$, i.e.:
$$
\mathbb{E} \bigg[ \int D_{\alpha}( P_{\theta}, P^0 ) \tilde{\pi}_{n,\alpha}^{\hat{K}}(d\theta|X_1^n) \bigg] \leq r_n.
$$
Then, using Markov's inequality for any $s_n$ such that $r_n=o(s_n)$ and $s_n  \rightarrow 0$ and any constant $M>0$:
$$
\mathbb{E} \bigg[ \tilde{\pi}_{n,\alpha}^{\hat{K}}\bigg(D_{\alpha}( P_{\theta}, P^0 )>M s_n|X_1^n\bigg) \bigg] \leq \frac{\mathbb{E} \bigg[ \int D_{\alpha}( P_{\theta}, P^0 ) \tilde{\pi}_{n,\alpha}^{\hat{K}}(d\theta|X_1^n) \bigg]}{M s_n} \leq \frac{r_n}{Ms_n} \xrightarrow[n\rightarrow +\infty]{} 0.
$$
Hence, we obtain the convergence in mean of $\tilde{\pi}_{n,\alpha}^{\hat{K}}\big(D_{\alpha}( P_{\theta}, P^0 )>M s_n|X_1^n\big)$ to $0$, which implies the convergence in probability of $\tilde{\pi}_{n,\alpha}^{\hat{K}}\big(D_{\alpha}( P_{\theta}, P^0 )>M s_n|X_1^n\big)$ to $0$, i.e. the concentration of $\tilde{\pi}_{n,\alpha}^{\hat{K}}(.|X_1^n)$ to $P^0$ at rate $s_n$.

\section{Proof of Theorem \ref{thm-true-model}.}\label{apd:proof-thm-one}

First, we need Donsker and Varadhan's famous variational formula. Refer for example to \cite{MR2483528} for a proof (Lemma 1.1.3).

\begin{lemma}
\label{thm-dv}
For any probability $\lambda$ on some measurable space $(\textbf{E},\mathcal{E})$ and any measurable function
  $h: \textbf{E} \rightarrow \mathbb{R}$ such that $\int{\rm e}^h  \rm{d}\lambda < \infty$,
  \begin{equation*}
    \log\int {\rm e}^h \mathrm{d}\lambda = \underset{\rho \in \mathcal{M}_{1}^+(\textbf{E})}{\sup} \bigg\{ \int h \mathrm{d}\rho - \textrm{KL}(\rho,\lambda) \bigg\},
  \end{equation*}
  with the convention
  $\infty-\infty =-\infty$. Moreover, if $h$ is upper-bounded on the
  support of $\lambda$, then the supremum  on the right-hand
  side is reached by the distribution of the form:
  \begin{equation*}
    \lambda_h(d\beta) =
    \frac{{\rm e}^{h(\beta)} }{\int{\rm e}^h \mathrm{d}\lambda} \lambda(\mathrm{d}\beta).
  \end{equation*}
\end{lemma}

We come back to the proof of Theorem \ref{thm-true-model}. We adapt the proof of Theorem 4.1 in \cite{cherief2018consistency}. 

\begin{proof}
For any $\alpha \in (0,1)$ and $\theta \in \Omega := \cup_{K \geq 1} \Theta_K$, using the definition of Renyi divergence and $D_\alpha(P^{\otimes n},R^{\otimes n})=nD_\alpha(P,R)$ as data are i.i.d.:
$$
\mathbb{E}\bigg[ \exp\bigg(-\alpha r_{n}(P_\theta,P^0) + (1-\alpha)n D_\alpha(P_\theta,P^0)\bigg) \bigg] = 1
$$
where $r_n(P_\theta,P^0)=\sum_{i=1}^n \log({P^0(X_i)}/{P_\theta(X_i)})$ is the negative log-likelihood ratio.
Then we integrate and use Fubini's theorem, 
$$
\mathbb{E}\bigg[ \int \exp\bigg(-\alpha r_{n}(P_\theta,P^0) + (1-\alpha)n D_\alpha(P_\theta,P^0) \bigg) \pi(d\theta) \bigg] = 1.
$$
Using Lemma \ref{thm-dv},
\begin{multline*}
\mathbb{E}\bigg[ \exp\bigg( \sup_{\rho \in \mathcal{M}_1^+(\Omega)} \bigg\{ \int \bigg( -\alpha r_{n}(P_\theta,P^0) + (1-\alpha)n D_\alpha(P_\theta,P^0) \bigg) \rho(d\theta) - \textrm{KL}(\rho,\pi) \bigg\} \bigg) \bigg] = 1.
\end{multline*}
Then, using Jensen's inequality, $$ 
\mathbb{E}\bigg[ \sup_{\rho \in \mathcal{M}_1^+(\Omega)} \bigg\{ \int \bigg( -\alpha r_{n}(P_\theta,P^0) + (1-\alpha)n D_\alpha(P_\theta,P^0) \bigg) \rho(d\theta) - \textrm{KL}(\rho,\pi) \bigg\} \bigg] \leq 0.
$$
Now, we consider $\tilde{\pi}^{\hat{K}}_{n,\alpha}(.|X_1^n)$ as a distribution on $\mathcal{M}_1^+(\Omega)$ with all its mass on $\Theta_{\hat{K}}$, $$ 
\mathbb{E}\bigg[ \int \bigg( -\alpha r_{n}(P_\theta,P^0) + (1-\alpha)n D_\alpha(P_\theta,P^0) \bigg) \tilde{\pi}_{n,\alpha}^{\hat{K}}(d\theta|X_1^n) - \textrm{KL}(\tilde{\pi}^{\hat{K}}_{n,\alpha}(.|X_1^n),\pi) \bigg] \leq 0.
$$
We use $\textrm{KL}(\tilde{\pi}^{\hat{K}}_{n,\alpha}(.|X_1^n),\pi)=\textrm{KL}(\tilde{\pi}^{\hat{K}}_{n,\alpha}(.|X_1^n),\Pi_{\hat{K}})+\log(\frac{1}{\pi_{\hat{K}}})$, and we rearrange terms:
\begin{multline*}
\mathbb{E}\bigg[ \int D_\alpha(P_\theta,P^0) \tilde{\pi}_{n,\alpha}^{\hat{K}}(d\theta|X_1^n) \bigg] \\ \leq 
\mathbb{E}\bigg[ \frac{\alpha}{1-\alpha} \int \frac{r_{n}(P_\theta,P^0)}{n} \tilde{\pi}_{n,\alpha}^{\hat{K}}(d\theta|X_1^n) + \frac{\textrm{KL}(\tilde{\pi}^{\hat{K}}_{n,\alpha}(.|X_1^n),\Pi_{\hat{K}})}{n(1-\alpha)}+\frac{\log(\frac{1}{\pi_{\hat{K}}})}{n(1-\alpha)} \bigg].
\end{multline*}
By definition of $\hat{K}$,
\begin{multline*}
\mathbb{E}\bigg[ \int D_\alpha(P_\theta,P^0) \tilde{\pi}_{n,\alpha}^{\hat{K}}(d\theta|X_1^n) \bigg] \\ \leq \mathbb{E}\bigg[  \inf_{K \geq 1} \bigg\{ \frac{\alpha}{1-\alpha} \int \frac{r_{n}(P_\theta,P^0)}{n} \tilde{\pi}_{n,\alpha}^{{K}}(d\theta|X_1^n) + \frac{\textrm{KL}(\tilde{\pi}^{{K}}_{n,\alpha}(.|X_1^n),\Pi_{{K}})}{n(1-\alpha)}+\frac{\log(\frac{1}{\pi_{{K}}})}{n(1-\alpha)} \bigg\} \bigg]
\end{multline*}
which gives 
\begin{multline*}
\mathbb{E}\bigg[ \int D_\alpha(P_\theta,P^0) \tilde{\pi}_{n,\alpha}^{\hat{K}}(d\theta|X_1^n) \bigg] \\
\leq \inf_{K \geq 1} \bigg\{ \mathbb{E}\bigg[ \frac{\alpha}{1-\alpha} \int \frac{r_{n}(P_\theta,P^0)}{n} \tilde{\pi}_{n,\alpha}^{{K}}(d\theta|X_1^n) + \frac{\textrm{KL}(\tilde{\pi}^{{K}}_{n,\alpha}(.|X_1^n),\Pi_{{K}})}{n(1-\alpha)}+\frac{\log(\frac{1}{\pi_{{K}}})}{n(1-\alpha)} \bigg] \bigg\}
\end{multline*}
and hence, by definition of $\tilde{\pi}^{{K}}_{n,\alpha}(.|X_1^n)$, \begin{multline*}
\mathbb{E}\bigg[ \int D_\alpha(P_\theta,P^0) \tilde{\pi}_{n,\alpha}^{\hat{K}}(d\theta|X_1^n) \bigg] \\ \leq \inf_{K \geq 1} \bigg\{ \mathbb{E}\bigg[ \inf_{\rho \in \mathcal{F}_K} \bigg\{ \frac{\alpha}{1-\alpha} \int \frac{r_{n}(P_\theta,P^0)}{n} \rho(d\theta) + \frac{\textrm{KL}(\rho,\Pi_{{K}})}{n(1-\alpha)} \bigg\} +\frac{\log(\frac{1}{\pi_{{K}}})}{n(1-\alpha)} \bigg] \bigg\}.
\end{multline*}
which leads to, \begin{multline*}
\mathbb{E}\bigg[ \int D_\alpha(P_\theta,P^0) \tilde{\pi}_{n,\alpha}^{\hat{K}}(d\theta|X_1^n) \bigg]\\  \leq \inf_{K \geq 1} \inf_{\rho \in \mathcal{F}_K} \bigg\{ \mathbb{E}\bigg[ \frac{\alpha}{1-\alpha} \int \frac{r_{n}(P_\theta,P^0)}{n} \rho(d\theta) + \frac{\textrm{KL}(\rho,\Pi_{{K}})}{n(1-\alpha)} +\frac{\log(\frac{1}{\pi_{{K}}})}{n(1-\alpha)} \bigg] \bigg\}.
\end{multline*}
Finally,
\begin{multline*}
\mathbb{E}\bigg[ \int D_\alpha(P_\theta,P^0) \tilde{\pi}_{n,\alpha}^{\hat{K}}(d\theta|X_1^n) \bigg] \\
\leq \inf_{K \geq 1} \bigg\{ \inf_{\rho_K \in \mathcal{F}_K} \bigg\{ \frac{\alpha}{1-\alpha} \int \textrm{KL}(P^0,P_{\theta_K}) \rho_K(d\theta_K) + \frac{\textrm{KL}(\rho_K,\Pi_K)}{n(1-\alpha)} \bigg\} +  \frac{\log(\frac{1}{\pi_K})}{n(1-\alpha)} \bigg\}.
\end{multline*}
The theorem is a direct corollary of this inequality as soon as Assumption \ref{priormass} is satisfied.

\end{proof}

\section{Proof of Theorem \ref{thm-extension}.}\label{apd:firsthalf}

\begin{proof}

Fix $\alpha \in (0,1)$ and let us prove Theorem \ref{thm-extension}. Let us recall that $\Theta_K(r_{K,n})$ is defined as the set of parameters $\theta^*_K \in \Theta_{K}$, for which there is a distribution $\rho_{K,n} \in \mathcal{F}_{K}$ such that:
\begin{equation*}
  \int \mathbb{E}\bigg[\log \frac{P_{\theta^*_K}(X_i)}{P_{\theta_{K}}(X_i)}\bigg] \rho_{K,n}(d\theta_{K}) \leq r_{K,n} \hspace{0.2cm} \text{and} \hspace{0.2cm} 
  \textrm{KL}(\rho_{K,n},\Pi_{K}) \leq n r_{K,n}.
\end{equation*}

We begin from:
\begin{multline*}
\mathbb{E}\bigg[ \int D_\alpha(P_\theta,P^0) \tilde{\pi}_{n,\alpha}^{\hat{K}}(d\theta|X_1^n) \bigg] \\
\leq \inf_{K \geq 1} \bigg\{ \inf_{\rho_K \in \mathcal{F}_K} \bigg\{ \frac{\alpha}{1-\alpha} \int \textrm{KL}(P^0,P_{\theta_K}) \rho_K(d\theta_K) + \frac{\textrm{KL}(\rho_K,\Pi_K)}{n(1-\alpha)} \bigg\} +  \frac{\log(\frac{1}{\pi_K})}{n(1-\alpha)} \bigg\}.
\end{multline*}
Then, we write for any $K$, any $\theta_K \in \Theta_K$, $\theta^*_K \in \Theta_K$:
$$
\textrm{KL}(P^0,P_{\theta_K}) = \textrm{KL}(P^0,P_{\theta_K^*}) + \mathbb{E}\bigg[\log \frac{P_{\theta^*_K}(X_i)}{P_{\theta_K}(X_i)}\bigg]
$$
which gives:
\begin{multline*}
\mathbb{E} \bigg[ \int D_{\alpha}( P_{\theta}, P^0 ) \tilde{\pi}_{n,\alpha}^{\hat{K}}(d\theta|X_1^n) \bigg]
\\
\leq \inf_{K \geq 1} \bigg\{ \inf_{\theta^*_K \in \Theta_K} \bigg\{ \frac{\alpha}{1-\alpha} \textrm{KL}(P^0,P_{\theta^*_K}) + \inf_{\rho_K \in \mathcal{F}_K} \bigg\{ \frac{\alpha}{1-\alpha} \int \mathbb{E}\bigg[\log \frac{P_{\theta^*_K}(X_i)}{P_{\theta_K}(X_i)}\bigg] \rho_K(d\theta_K) \\ 
+ \frac{\textrm{KL}(\rho_K,\Pi_K)}{n(1-\alpha)} \bigg\} \bigg\} +  \frac{\log(\frac{1}{\pi_K})}{n(1-\alpha)} \bigg\}.
\end{multline*}
Hence, using the definition of $\Theta_{K}(r_{K,n})$ and upper bounding the right-hand-side of the previous inequality by an inf over $\Theta_{K}(r_{K,n})$, we conclure:
\begin{equation*}
\mathbb{E} \bigg[ \int D_{\alpha}( P_{\theta}, P^0 ) \tilde{\pi}_{n,\alpha}^{\hat{K}}(d\theta|X_1^n) \bigg] \leq \inf_{K \geq 1} \bigg\{ \frac{\alpha}{1-\alpha} \inf_{\theta^* \in \Theta_K(r_{K,n})} \textrm{KL}(P^0,P_{\theta^*_K}) + \frac{1+\alpha}{1-\alpha} r_{K,n} +  \frac{\log(\frac{1}{\pi_{K}})}{n(1-\alpha)} \bigg\}.
\end{equation*}

\end{proof}

\section{Proof of Theorem \ref{corPCA}.}\label{apd:second}

\begin{proof}

We still consider the framework of probabilistic PCA in Section \ref{sec:example}. We assume that there exists a true rank $K_0$ and a matrix $\textit{\textbf{W}}_0 \in \mathbb{R}^{d\times K_0}$ with bounded coefficients such that the true distribution of each $X_i$ is $\mathcal{N}(0,\textit{\textbf{W}}_0\textit{\textbf{W}}_0^T + \sigma^2 \textit{\textbf{I}}_d)$, and we place a prior $\Pi_{K_0}=\mathcal{N}(0,s^2\textit{\textbf{I}}_d)^{\otimes K_0}$ and a variational approximation $\rho_{K_0}=\rho^{\otimes K_0}$ on $W$ given $K=K_0$ where we denote $\rho=\mathcal{N}(0,\frac{1}{dn^2}\textit{\textbf{I}}_d)$. We recall that $\pi_K=\frac{1}{d}$ for any $K=1,...,d$.

To obtain the rate of convergence $r_n=dK_0\log(nd)/n$ for probabilistic PCA, we just need to show that the quantities in Assumption \ref{priormass} are upper bounded by $r_n$ (up to a constant) as we have $\log(1/\pi_{K_0})/n$ much smaller than $r_n$:
$$
\int \textrm{KL}\bigg(\mathcal{N}(0,\textit{\textbf{W}}_0\textit{\textbf{W}}_0^T + \sigma^2 \textit{\textbf{I}}_d),\mathcal{N}(0,\textit{\textbf{W}}\textit{\textbf{W}}^T + \sigma^2 \textit{\textbf{I}}_d)\bigg) \rho_{K_0}(d\theta_K) \hspace{0.3cm} , \hspace{0.3cm} \frac{\textrm{KL}(\rho_{K_0},\Pi_{K_0})}{n}.
$$

We have two terms. The first one, i.e. the Kullback-Leibler term, provides a rate of convergence of $dK_0\log(dn)/n$ as:
\begin{align*}
    \textrm{KL}(\rho_{K_0},\Pi_{K_0}) & = \sum_{j=1}^{K_0} \textrm{KL}\bigg(\mathcal{N}(0,\frac{1}{dn^2} \textit{\textbf{I}}_d),\mathcal{N}(0,s^2 \textit{\textbf{I}}_d)\bigg) \\
    & = \frac{K_0}{2} \bigg( \frac{1}{n^2s^2} -d +d\log(s^2) + d\log(dn^2) \bigg) \\
    & \leq \frac{K_0}{2n^2s^2} - \frac{dK_0}{2} + \frac{dK_0\log(s^2)}{2} + dK_0\log(dn).
\end{align*}

The integral is much more complicated to deal with. We will show that it leads to a rate faster than $dK_0\log(dn)/n$. If we denote $\mathbb{E}$ the expectation with respect to $\rho_{K_0}$, then the integral will be equal to:
$$
\frac{1}{2} \mathbb{E}\bigg[\text{Tr}\bigg((\textit{\textbf{W}}\textit{\textbf{W}}^T+\sigma^2\textit{\textbf{I}}_d)^{-1}(\textit{\textbf{W}}_0\textit{\textbf{W}}_0^T+\sigma^2\textit{\textbf{I}}_d)\bigg)\bigg] - \frac{d}{2} + \frac{1}{2} \mathbb{E}\bigg[\log\bigg(\frac{\det(\textit{\textbf{W}}\textit{\textbf{W}}^T+\sigma^2\textit{\textbf{I}}_d)}{\det(\textit{\textbf{W}}_0\textit{\textbf{W}}_0^T+\sigma^2\textit{\textbf{I}}_d)}\bigg)\bigg].
$$

The expectation of the log-ratio is easy to upper bound. We denote $\lambda_1,...,\lambda_d$ the positive eigenvalues of the positive definite matrix $\textit{\textbf{W}}_0\textit{\textbf{W}}_0^T+\sigma^2\textit{\textbf{I}}_d$. Then for each $j=1,...,d$, $\lambda_j \geq \sigma^2$ and using Jensen's inequality and the log-concavity of the determinant:
\begin{align*}
    \mathbb{E}\bigg[\log\bigg(\det(\textit{\textbf{W}}\textit{\textbf{W}}^T+\sigma^2\textit{\textbf{I}}_d)\bigg)\bigg] & \leq \log\bigg(\det\big(\mathbb{E}[\textit{\textbf{W}}\textit{\textbf{W}}^T]+\sigma^2\textit{\textbf{I}}_d\big)\bigg) \\
    & = \log\bigg(\det\big(\textit{\textbf{W}}_0\textit{\textbf{W}}_0^T+\sigma^2\textit{\textbf{I}}_d+\frac{1}{dn^2}\textit{\textbf{I}}_d\big)\bigg) \\
    & = \sum_{j=1}^d \log\bigg(\lambda_j+\frac{1}{dn^2}\bigg) \\
    & = \sum_{j=1}^d \log(\lambda_j) + \sum_{j=1}^d \log\bigg(1+\frac{1}{\lambda_jdn^2}\bigg) \\
    & = \mathbb{E}\bigg[\log\bigg(\det(\textit{\textbf{W}}_0\textit{\textbf{W}}_0^T+\sigma^2\textit{\textbf{I}}_d)\bigg)\bigg] + \sum_{j=1}^d \log\bigg(1+\frac{1}{\lambda_jdn^2}\bigg) \\
    & \leq \mathbb{E}\bigg[\log\bigg(\det(\textit{\textbf{W}}_0\textit{\textbf{W}}_0^T+\sigma^2\textit{\textbf{I}}_d)\bigg)\bigg] + \sum_{j=1}^d \frac{1}{\lambda_jdn^2} \\
    & \leq \mathbb{E}\bigg[\log\bigg(\det(\textit{\textbf{W}}_0\textit{\textbf{W}}_0^T+\sigma^2\textit{\textbf{I}}_d)\bigg)\bigg] + \frac{1}{n^2 \sigma^2}
\end{align*}
and then the expectation of the log-ratio provides a rate of convergence of $1/n^2$:
$$
\mathbb{E}\bigg[\log\bigg(\frac{\det(\textit{\textbf{W}}\textit{\textbf{W}}^T+\sigma^2\textit{\textbf{I}}_d)}{\det(\textit{\textbf{W}}_0\textit{\textbf{W}}_0^T+\sigma^2\textit{\textbf{I}}_d)}\bigg)\bigg] \leq \frac{1}{n^2 \sigma^2}.
$$

The remainder can be bounded as follows:

\begin{align*}
    \mathbb{E}\bigg[\text{Tr}\bigg((\textit{\textbf{W}}\textit{\textbf{W}}^T+\sigma^2\textit{\textbf{I}}_d)^{-1}(\textit{\textbf{W}}_0 & \textit{\textbf{W}}_0^T+\sigma^2\textit{\textbf{I}}_d)\bigg)\bigg] - d  \\
    & = \mathbb{E}\bigg[\text{Tr}\bigg((\textit{\textbf{W}}\textit{\textbf{W}}^T+\sigma^2\textit{\textbf{I}}_d)^{-1}(\textit{\textbf{W}}_0\textit{\textbf{W}}_0^T-\textit{\textbf{W}}\textit{\textbf{W}}^T)\bigg)\bigg] \\
    & \leq \mathbb{E}\bigg[ \| (\textit{\textbf{W}}\textit{\textbf{W}}^T+\sigma^2\textit{\textbf{I}}_d)^{-1} \|_F \times  \|\textit{\textbf{W}}_0\textit{\textbf{W}}_0^T-\textit{\textbf{W}}\textit{\textbf{W}}^T\|_F \bigg] \\
    & \leq \sqrt{d} \mathbb{E}\bigg[ \| (\textit{\textbf{W}}\textit{\textbf{W}}^T+\sigma^2\textit{\textbf{I}}_d)^{-1} \|_2 \times  \|\textit{\textbf{W}}_0\textit{\textbf{W}}_0^T-\textit{\textbf{W}}\textit{\textbf{W}}^T\|_F \bigg] \\
    & = \sqrt{d} \mathbb{E}\bigg[ \sigma_{\max}\big((\textit{\textbf{W}}_0\textit{\textbf{W}}_0^T+\sigma^2\textit{\textbf{I}}_d)^{-1}\big) \times  \|\textit{\textbf{W}}_0\textit{\textbf{W}}_0^T-\textit{\textbf{W}}\textit{\textbf{W}}^T\|_F \bigg] \\
    & = \sqrt{d} \mathbb{E}\bigg[ \sigma_{\min}(\textit{\textbf{W}}_0\textit{\textbf{W}}_0^T+\sigma^2\textit{\textbf{I}}_d)^{-1} \times  \|\textit{\textbf{W}}_0\textit{\textbf{W}}_0^T-\textit{\textbf{W}}\textit{\textbf{W}}^T\|_F \bigg]\end{align*}
i.e.
\begin{align*}
    \mathbb{E}\bigg[\text{Tr}\bigg((\textit{\textbf{W}}\textit{\textbf{W}}^T+\sigma^2\textit{\textbf{I}}_d)^{-1}(\textit{\textbf{W}}_0 \textit{\textbf{W}}_0^T+\sigma^2\textit{\textbf{I}}_d)\bigg)\bigg] - d  & \leq \sqrt{d} \mathbb{E}\bigg[ (\sigma^2)^{-1} \times  \|\textit{\textbf{W}}_0\textit{\textbf{W}}_0^T-\textit{\textbf{W}}\textit{\textbf{W}}^T\|_F \bigg] \\
    & = \frac{\sqrt{d}}{\sigma^2} \mathbb{E}\bigg[ \|\textit{\textbf{W}}_0\textit{\textbf{W}}_0^T-\textit{\textbf{W}}\textit{\textbf{W}}^T\|_F \bigg]
\end{align*}
where $\|.\|_F$ is the Frobenius norm on matrices, $\|.\|_2$ the spectral norm, and $\sigma_{\min}(\textit{\textbf{A}})$, $\sigma_{\max}(\textit{\textbf{A}})$ the lowest and largest singular values of a matrix $\textit{\textbf{A}}$. We use the fact that for a symmetric semi-definite positive matrix: $\sigma_{\max}(\textit{\textbf{A}}^{-1}) = \big(\sigma_{\min}(\textit{\textbf{A}})\big)^{-1}$ and $\sigma_{\min}(\textit{\textbf{A}}+\sigma^2 \textit{\textbf{I}}_d) \geq \sigma^2$, as well as the inequality $\|\textit{\textbf{A}}\|_F \leq \sqrt{d} \|\textit{\textbf{A}}\|_2$ for any $d \times d$ matrix $\textit{\textbf{A}}$.

The only thing left to do is to upper bound the expectation of the Frobenius norm of $\textit{\textbf{W}}_0\textit{\textbf{W}}_0^T-\textit{\textbf{W}}\textit{\textbf{W}}^T$ by a multiple of $\frac{\sqrt{d}K_0\log(dn)}{n}$. We use the triangle and Cauchy-Schwarz's inequalities:
\begin{align*}
    \mathbb{E}\bigg[ \|\textit{\textbf{W}}_0\textit{\textbf{W}}_0^T-\textit{\textbf{W}}\textit{\textbf{W}}^T\|_F \bigg] & \leq \mathbb{E}\bigg[ \|\textit{\textbf{W}}\textit{\textbf{W}}^T-\textit{\textbf{W}}\textit{\textbf{W}}_0^T\|_F \bigg] + \mathbb{E}\bigg[ \|\textit{\textbf{W}}\textit{\textbf{W}}_0^T-\textit{\textbf{W}}_0\textit{\textbf{W}}_0^T\|_F \bigg] \\
    & \leq \mathbb{E}\bigg[ \|\textit{\textbf{W}}(\textit{\textbf{W}}-\textit{\textbf{W}}_0)^T\|_F \bigg] + \mathbb{E}\bigg[ \|(\textit{\textbf{W}}-\textit{\textbf{W}}_0)\textit{\textbf{W}}_0^T\|_F \bigg] \\
    & \leq \mathbb{E}\bigg[ \|\textit{\textbf{W}}\|_F \|\textit{\textbf{W}}-\textit{\textbf{W}}_0\|_F \bigg] + \mathbb{E}\bigg[ \|\textit{\textbf{W}}-\textit{\textbf{W}}_0\|_F \|\textit{\textbf{W}}_0\|_F \bigg] \\ 
    & \leq \sqrt{ \mathbb{E}\big[ \|\textit{\textbf{W}}\|_F^2\big] \mathbb{E}\big[\|\textit{\textbf{W}}-\textit{\textbf{W}}_0\|_F^2 \big] } + \sqrt{ \mathbb{E}\big[ \|\textit{\textbf{W}}-\textit{\textbf{W}}_0\|_F^2\big] \mathbb{E}\big[\|\textit{\textbf{W}}_0\|_F^2 \big] } \\ 
    & \leq \sqrt{ \mathbb{E}\big[ \|\textit{\textbf{W}}\|_F^2\big] \mathbb{E}\big[\|\textit{\textbf{W}}-\textit{\textbf{W}}_0\|_F^2 \big] } + \|\textit{\textbf{W}}_0\|_F \sqrt{ \mathbb{E}\big[ \|\textit{\textbf{W}}-\textit{\textbf{W}}_0\|_F^2\big] }.
\end{align*}

We can upper bound $\|\textit{\textbf{W}}_0\|_F=\sqrt{\sum_{i=1}^d \sum_{j=1}^{K_0} (\textit{\textbf{W}}_0)_{i,j}^2}$ by $\sqrt{dK_0}C$ where $C$ is an upper bound on each of the coefficients of matrix $\textit{\textbf{W}}_0$.

Also, we can notice that $dn^2 \|\textit{\textbf{W}}-\textit{\textbf{W}}_0\|_F^2 = \sum_{i=1}^d \sum_{j=1}^{K_0} \big(\sqrt{d}n(\textit{\textbf{W}}_{i,j}-(\textit{\textbf{W}}_0)_{i,j})\big)^2$ is a sum of squares of independent standard normal random variables. Thus $dn^2 \|\textit{\textbf{W}}-\textit{\textbf{W}}_0\|_F^2 $ follows a chi-squared distribution with $dK_0$ degrees of freedom and its expectation is equal to $dK_0$. Hence:
$$
\mathbb{E}\big[\|\textit{\textbf{W}}-\textit{\textbf{W}}_0\|_F^2 \big]=\frac{K_0}{n^2}.
$$

Similarly, as $\textit{\textbf{W}}_{i,j}-(\textit{\textbf{W}}_0)_{i,j}$ is centered, we get:
\begin{align*}
    \mathbb{E}\big[ \|\textit{\textbf{W}}\|_F^2\big] & = \mathbb{E} \bigg[\sum_{i=1}^d \sum_{j=1}^{K_0} \textit{\textbf{W}}_{i,j}^2 \bigg] \\
    & = \sum_{i=1}^d \sum_{j=1}^{K_0} \mathbb{E} \bigg[ \big(\textit{\textbf{W}}_{i,j}-(\textit{\textbf{W}}_0)_{i,j}\big)^2 + (\textit{\textbf{W}}_0)_{i,j}^2 - 2 (\textit{\textbf{W}}_0)_{i,j} \big(\textit{\textbf{W}}_{i,j}-(\textit{\textbf{W}}_0)_{i,j}\big) \bigg] \\
    & = \mathbb{E}\big[\|\textit{\textbf{W}}-\textit{\textbf{W}}_0\|_F^2 \big] + \|\textit{\textbf{W}}_0\|_F^2 \\
    & \leq \frac{K_0}{n^2} + dK_0C^2 \\
    & = \bigg(dC^2+\frac{1}{n^2}\bigg)K_0.
\end{align*}

Thus, we obtain:
\begin{align*}
    \mathbb{E}\bigg[ \|\textit{\textbf{W}}_0\textit{\textbf{W}}_0^T-\textit{\textbf{W}}\textit{\textbf{W}}^T\|_F \bigg] & \leq \frac{\sqrt{K_0}}{n} \sqrt{K_0}\sqrt{dC^2+\frac{1}{n^2}} + \sqrt{dK_0}C \frac{\sqrt{K_0}}{n} \\
    & = \frac{K_0}{n}\sqrt{dC^2+\frac{1}{n^2}} + \frac{\sqrt{d}K_0C}{n} \\
    & \leq \frac{K_0}{n}\bigg(\sqrt{d}C+\frac{1}{n}\bigg) + \frac{\sqrt{d}K_0C}{n} \\
    & = \frac{K_0}{n}\bigg(2\sqrt{d} C+\frac{1}{n}\bigg).
\end{align*}

Hence, the order of the upper bound of the expectation of the Fobrenius norm of matrix $\textit{\textbf{W}}_0\textit{\textbf{W}}_0^T-\textit{\textbf{W}}\textit{\textbf{W}}^T$ is $\frac{\sqrt{d}K_0}{n}<\frac{\sqrt{d}K_0\log(dn)}{n}$.

Finally, the consistency rate associated with the integral term is $\frac{dK_0}{n}$, and the overall rate of convergence is $\frac{dK_0\log(dn)}{n}$.
\end{proof}

\section{Computation of the ELBO for probabilistic PCA.}\label{apd:ELBOcomp}

We consider the framework of probabilistic PCA detailed in Section \ref{sec:example}. Given rank $K$, we place independent Gaussian priors on the columns $W_1,...,W_K$ of $\textit{\textbf{W}}$ such that $\Pi_K=\mathcal{N}(0,s^2\textit{\textbf{I}}_d)^{\otimes K}$, and Gaussian independent variational approximations $\mathcal{N}(\mu_j,{\bold{\Sigma}}_j)$ for the columns of $\textit{\textbf{W}}$. The ELBO associated with rank $K$ and variational approximation $\rho_K=\otimes_{j=1}^{K} \mathcal{N}(\mu_j,{\bold{\Sigma}}_j)$ is given by:
\[
\textrm{ELBO}_K(\rho_K) = \alpha \int \ell_n(\textit{\textbf{W}}) \rho_K(d\textit{\textbf{W}}) - \textrm{KL}\big(\rho_K,\Pi_K\big).
\]

The Kullback-Leibler term $\textrm{KL}\big(\rho_K,\Pi_K\big)$ is equal to:
$$
\frac{1}{2} \sum_{j=1}^K \bigg\{ \frac{\text{Tr}({\bold{\Sigma}}_j)}{s^2} + \frac{\mu_j^T\mu_j}{s^2} - \log\big(\det({\bold{\Sigma}}_j)\big) \bigg\} - \frac{dK}{2} + \frac{dK\log(s^2)}{2}
$$
while the average log-likelihood $\int \ell_n(\textit{\textbf{W}}) \rho_K(d\textit{\textbf{W}})$ is:
$$
-\frac{dn}{2} \log(2\pi) - \frac{n}{2} \int \log\big(\det(\textit{\textbf{W}}\textit{\textbf{W}}^T+\sigma^2\textit{\textbf{I}}_d)\big) \rho_K(dW) - \frac{1}{2} \sum_{i=1}^{n} \int X_i^T (\textit{\textbf{W}}\textit{\textbf{W}}^T+\sigma^2\textit{\textbf{I}}_d)^{-1} X_i \hspace{0.1cm} \rho_K(d\textit{\textbf{W}})
$$
where both integrals can be computed thanks to Monte-Carlo sampling approximations:
$$
\int \log\big(\det(\textit{\textbf{W}}\textit{\textbf{W}}^T+\sigma^2\textit{\textbf{I}}_d)\big) \rho_K(d\textit{\textbf{W}}) \approx \sum_{\ell=1}^N \log\big(\det(\textit{\textbf{W}}^{(\ell)} \textit{\textbf{W}}^{(\ell) T}+\sigma^2\textit{\textbf{I}}_d)\big)
$$
and
$$
\int X_i^T (\textit{\textbf{W}}\textit{\textbf{W}}^T+\sigma^2\textit{\textbf{I}}_d)^{-1} X_i \hspace{0.1cm} \rho_K(d\textit{\textbf{W}}) \approx \sum_{\ell=1}^N X_i^T (\textit{\textbf{W}}^{(\ell)} \textit{\textbf{W}}^{(\ell) T}+\sigma^2\textit{\textbf{I}}_d)^{-1} X_i
$$
where $\textit{\textbf{W}}^{(1)},...,\textit{\textbf{W}}^{(N)}$ are $N$ i.i.d. data sampled from $\rho_K$.

The inverse matrix $(\textit{\textbf{W}}\textit{\textbf{W}}^T+\sigma^2\textit{\textbf{I}}_d)^{-1}$ can be derived thanks to classical inversion algorithms. For instance, it is possible to do so in $\mathcal{O} (Kd^2) $ operations instead of the classical $\mathcal{O}(d^3)$ inversion procedure thanks to Sherman-Morrison formula: for any matrix $\textit{\textbf{A}} \in \mathbb{R}^{d\times d}$ and vectors $u,v \in \mathbb{R}^d$ such that $\textit{\textbf{A}}+uv^T$ is invertible,
$$
(\textit{\textbf{A}}+uv^T)^{-1} = \textit{\textbf{A}}^{-1} - \frac{\textit{\textbf{A}}^{-1}uv^T\textit{\textbf{A}}^{-1}}{1+v^T\textit{\textbf{A}}^{-1}u}.
$$
We write
$$ 
\textit{\textbf{W}}\textit{\textbf{W}}^T+\sigma^2\textit{\textbf{I}}_d=\sigma^2 \textit{\textbf{I}}_d + \sum_{j=1}^K W_jW_j^T = \bigg( \sigma^2 \textit{\textbf{I}}_d + \sum_{j=1}^{K-1} W_jW_j^T \bigg) + W_K W_K^T
$$ 
and iterate $K$ times Sherman-Morrison formula. The first time, we apply it to $\textit{\textbf{A}}=\sigma^2 \textit{\textbf{I}}_d + \sum_{j=1}^{K-1} W_jW_j^T$ and $u=v=\textit{\textbf{W}}_K$, then to $\textit{\textbf{A}}=\sigma^2 \textit{\textbf{I}}_d + \sum_{j=1}^{K-2} W_jW_j^T$ and $u=v=W_{K-1}$, and so on. We finally obtain $(\textit{\textbf{W}}\textit{\textbf{W}}^T+\sigma^2\textit{\textbf{I}}_d)^{-1}=\textit{\textbf{M}}_K$ where:

$$
\left\{
\begin{array}{l}
  \textit{\textbf{M}}_0 = \sigma^2 \textit{\textbf{I}}_d \\
  \forall j=1,...,K, \hspace{0.2cm}  \textit{\textbf{M}}_j = \textit{\textbf{M}}_{j-1} - \frac{1}{1+W_j^TZ_j} Z_jZ_j^T \text{  with  } Z_j=\textit{\textbf{M}}_{j-1}W_j.
\end{array}
\right.
$$

In order to compute the maximum value $\textrm{ELBO}(K)$ of the ELBO associated with rank $K$, one can use a stochastic gradient descent on $(\mu_1,{\bold{\Sigma}}_1,...,\mu_K,{\bold{\Sigma}}_K)$ that will converge to a local maximum and will give the variational estimator for rank $K$. Then, maximizing $\textrm{ELBO}(K)$ over desired values of $K$ leads to the optimal number of principal components and to the associated optimal variational approximation.

\section{Results in matrix norm for probabilistic PCA.}\label{apd:clip}

To prove Corollaries \ref{corPCAclip} and \ref{corPCAclipfreq}, we need the two lemmas presented behind. We introduce some notations first. We refer the interested reader to \cite{MatrixBook} for more details.

\textbf{Notations :} Let us call  $\mathcal{S}_d^+$ the set of $d\times d$ symmetric positive semi-definite matrices, and $\mathcal{X}_M=\big\{ \textit{\textbf{A}} \in  \mathcal{S}_d^+ /  \|\textit{\textbf{A}}\|_2 \leq M \big\}$. We define the vectorization of Matrix $\textit{\textbf{A}} \in \mathbb{R}^{p \times q}$ with columns $X_1,...,X_q$:
$$
\textnormal{Vec}(\textit{\textbf{A}}) = (\textit{\textbf{A}}_1^T,...,\textit{\textbf{A}}_q^T)^T \in \mathbb{R}^{p \times q}.
$$
We define the Frobenius inner product of two matrices $\textit{\textbf{A}} \in \mathbb{R}^{p\times q}$ and $\tilde{\textit{\textbf{A}}} \in \mathbb{R}^{p\times q}$, that is the sum of componentwise products:
$$
\textit{\textbf{A}} \cdot \tilde{\textit{\textbf{A}}} = \textnormal{Vec}(\textit{\textbf{A}})^T \textnormal{Vec}(\tilde{\textit{\textbf{A}}}).
$$
Notice that $\|\textit{\textbf{A}}\|_F^2 = \textit{\textbf{A}} \cdot \textit{\textbf{A}} = \textnormal{Vec}(\textit{\textbf{A}})^T \textnormal{Vec}(\textit{\textbf{A}})$.

We also introduce the Kronecker and Box products of two matrices $\textit{\textbf{A}} \in \mathbb{R}^{p_1\times q_1}$ and $\tilde{\textit{\textbf{A}}} \in \mathbb{R}^{p_2\times q_2}$ which are respectively the matrices $\textit{\textbf{A}} \otimes \tilde{\textit{\textbf{A}}} \in \mathbb{R}^{p_1p_2\times q_1q_2}$ and $\textit{\textbf{A}} \boxtimes \tilde{\textit{\textbf{A}}} \in \mathbb{R}^{p_1p_2\times q_1q_2}$ such that their coefficients are defined as:
$$
( \textit{\textbf{A}} \otimes \tilde{\textit{\textbf{A}}} )_{p_2(i-1)+j,q_2(k-1)+l} = \textit{\textbf{A}}_{i,k} \tilde{\textit{\textbf{A}}}_{j,l},
$$
\vspace{-0.7cm}
$$
( \textit{\textbf{A}} \boxtimes \tilde{\textit{\textbf{A}}} )_{p_2(i-1)+j,q_1(k-1)+l} = \textit{\textbf{A}}_{i,l} \tilde{\textit{\textbf{A}}}_{j,k}
$$
for any integers $i,j,k,l$ such that $1 \leq i \leq p_1$, $1 \leq j \leq q_1$, $1 \leq k \leq p_2$, $1 \leq l \leq q_2$.

We have the following properties for any matrix $\textit{\textbf{P}}$:
$$
( \textit{\textbf{A}} \otimes \tilde{\textit{\textbf{A}}} ) \textnormal{Vec}(\textit{\textbf{P}}) = \textnormal{Vec}(\tilde{\textit{\textbf{A}}} \textit{\textbf{P}} \textit{\textbf{A}}^T),
$$
\vspace{-0.7cm}
$$ 
( \textit{\textbf{A}} \boxtimes \tilde{\textit{\textbf{A}}} ) \textnormal{Vec}(\textit{\textbf{P}}) = \textnormal{Vec}(\tilde{\textit{\textbf{A}}} \textit{\textbf{P}}^T \textit{\textbf{A}}^T).
$$

We also define the gradient $\nabla f(\textit{\textbf{A}}) \in \mathbb{R}^{p\times q}$ and the Hessian $\nabla^2 f(\textit{\textbf{A}}) \in \mathbb{R}^{pq\times pq}$ of a differentiable function $f: \mathbb{R}^{p\times q} \rightarrow \mathbb{R}$ at matrix $\textit{\textbf{A}}$:
$$
( \nabla f(\textit{\textbf{A}}) )_{p_2(i-1)+j,q_2(k-1)+l} = \frac{\partial f(\textit{\textbf{A}})}{\partial\textit{\textbf{A}}_{i,j}},
$$
\vspace{-0.4cm}
$$
( \nabla^2 f(\textit{\textbf{A}}) )_{p_2(j-1)+i,p_2(l-1)+k} = \frac{\partial^2 f(\textit{\textbf{A}})}{\partial\textit{\textbf{A}}_{i,j} \partial\textit{\textbf{A}}_{k,l}}
$$
for any integers $i,j,k,l$ such that $1 \leq i,k \leq p$, $1 \leq j,l \leq q$ where $\partial f$ is the partial derivative of $f$.

We say that a differentiable function $f: \mathbb{R}^{p\times q} \rightarrow \mathbb{R}$ is $s$-strongly convex in $\mathcal{S} \subset \mathbb{R}^{pq\times pq}$ with respect to the norm $\|.\|$ as soon as one of the two following equivalent properties is satisfied:
$$
f(\textit{\textbf{A}}) \geq f(\tilde{\textit{\textbf{A}}}) + \nabla f(\textit{\textbf{A}}) \cdot ( \textit{\textbf{A}} - \tilde{\textit{\textbf{A}}}) + \frac{s}{2} \| \textit{\textbf{A}} - \tilde{\textit{\textbf{A}}} \|^2
$$
or
$$
\textnormal{Vec}(\textit{\textbf{P}})^T \nabla^2 f(\textit{\textbf{A}}) \textnormal{Vec}(\textit{\textbf{P}}) \geq s \| \textit{\textbf{P}} \|^2
$$
for any matrix $\textit{\textbf{A}}, \tilde{\textit{\textbf{A}}} \in \mathcal{S}$ and any symmetric matrix $\textit{\textbf{P}} \in \mathbb{R}^{pq\times pq}$.

\begin{lemma}
\label{corCvx}
Then, function $f:\textit{\textbf{A}} \rightarrow -\log\big(\det(\textit{\textbf{A}}+M \textit{\textbf{I}}_d)\big)$ is $1/(M+\sigma^2)^{2}$ strongly convex in $\mathcal{X}_M$ with respect to the Frobenius norm.
\end{lemma}

\begin{proof}
The proof follows the same steps than the proof of Theorem 3.1 in \cite{OnlineLogDet}.

The Hessian of function $f$ at any symmetric matrix in $\textit{\textbf{A}} \in \mathcal{X}_M$ is given by (see \cite{MatrixBook}):
$$
\nabla^2 f(\textit{\textbf{A}}) = \bigg((\textit{\textbf{A}}+M \textit{\textbf{I}}_d)^{-1}\bigg)^T \boxtimes (\textit{\textbf{A}}+M \textit{\textbf{I}}_d)^{-1} = (\textit{\textbf{A}}+M \textit{\textbf{I}}_d)^{-1} \boxtimes (\textit{\textbf{A}}+M \textit{\textbf{I}}_d)^{-1}.
$$
Then, we have for any $\textit{\textbf{A}} \in \mathcal{X}_M$ and any symmetric matrix $\textit{\textbf{P}} \in \mathbb{R}^{pq\times pq}$:
\begin{align*}
\textnormal{Vec}(\textit{\textbf{P}})^T \nabla^2 f(\textit{\textbf{A}}) \textnormal{Vec}(\textit{\textbf{P}})
& = \textnormal{Vec}(\textit{\textbf{P}})^T \bigg( (\textit{\textbf{A}}+M \textit{\textbf{I}}_d)^{-1} \boxtimes (\textit{\textbf{A}} +M \textit{\textbf{I}}_d)^{-1} \bigg) \textnormal{Vec}(\textit{\textbf{P}}) \\
& = \textnormal{Vec}(\textit{\textbf{P}})^T \textnormal{Vec}\bigg((\textit{\textbf{A}}+M \textit{\textbf{I}}_d)^{-1} \textit{\textbf{P}}^T (\textit{\textbf{A}}+M \textit{\textbf{I}}_d)^{-1}\bigg) \\
& = \textnormal{Vec}(\textit{\textbf{P}})^T \textnormal{Vec}\bigg((\textit{\textbf{A}}+M \textit{\textbf{I}}_d)^{-1} \textit{\textbf{P}} (\textit{\textbf{A}}+M \textit{\textbf{I}}_d)^{-1}\bigg) \\
& = \textnormal{Vec}(\textit{\textbf{P}})^T \bigg( (\textit{\textbf{A}}+M \textit{\textbf{I}}_d)^{-1} \otimes (\textit{\textbf{A}}+M \textit{\textbf{I}}_d)^{-1} \bigg) \textnormal{Vec}(\textit{\textbf{P}}).
\end{align*}

Note that the eigenvalues of a Kronecker product $\textit{\textbf{A}} \otimes \textit{\textbf{P}}$ are the products of an eigenvalue of $\textit{\textbf{A}}$ and an eigenvalue of $\textit{\textbf{P}}$, and the eigenvalues of $\textit{\textbf{P}}^{-1}$ are the inverse of the eigenvalues of $\textit{\textbf{P}}$. Moreover, the maximum eigenvalue of $\textit{\textbf{A}}+M \textit{\textbf{I}}_d$ is $\|\textit{\textbf{A}}\|_2+\sigma^2$,  so the minimum eigenvalue of $(\textit{\textbf{A}}+M \textit{\textbf{I}}_d)^{-1} \otimes (\textit{\textbf{A}}+M \textit{\textbf{I}}_d)^{-1}$ is equal to $(\|\textit{\textbf{A}}\|_2+\sigma^2)^{-2}$. Hence, for any matrix $\textit{\textbf{A}} \in \mathcal{X}_M$, we get:
\begin{align*}
\textnormal{Vec}(\textit{\textbf{P}})^T \bigg( (\textit{\textbf{A}}+M \textit{\textbf{I}}_d)^{-1} \otimes (\textit{\textbf{A}}+M \textit{\textbf{I}}_d)^{-1} \bigg) \textnormal{Vec}(\textit{\textbf{P}})^T & \geq  (\|\textit{\textbf{A}}\|_2+\sigma^2)^{-2} \textnormal{Vec}(\textit{\textbf{P}})^T \textnormal{Vec}(\textit{\textbf{P}}) \\
& \geq \frac{1}{(M+\sigma^2)^{2}}  \textnormal{Vec}(\textit{\textbf{P}})^T \textnormal{Vec}(\textit{\textbf{P}}),
\end{align*}
and we conclude using the definition of the strong convexity and $\|\textit{\textbf{P}}\|_F^2 =\textnormal{Vec}(\textit{\textbf{P}})^T \textnormal{Vec}(\textit{\textbf{P}})$.

\end{proof}

\begin{lemma}
\label{corStrongCvx}
For any $\alpha \in (0,1)$ and any matrices $\textit{\textbf{W}} \in \mathbb{R}^{d\times K_1}$ and $\tilde{\textit{\textbf{W}}} \in \mathbb{R}^{d\times K_2}$, as soon as the spectral norms of $\textit{\textbf{W}}\textit{\textbf{W}}^T$ and $\tilde{\textit{\textbf{W}}}\tilde{\textit{\textbf{W}}}^T$ are bounded by a constant $B^2$, then:
\begin{equation*}
 D_{\alpha}( P_{\textit{\textbf{W}}}, P_{\tilde{\textit{\textbf{W}}}} ) \geq \frac{\alpha}{16(B^2+\sigma^2)^2} \big\| \tilde{\textit{\textbf{W}}}\tilde{\textit{\textbf{W}}}^T-\textit{\textbf{W}}\textit{\textbf{W}}^T \big\|_F^2.
\end{equation*}
\end{lemma}

\begin{proof}

We recall that function $f:\textit{\textbf{A}} \rightarrow -\log\big(\det(\textit{\textbf{A}}+M \textit{\textbf{I}}_d)\big)$ is $1/(M+\sigma^2)^2$ strongly convex in $\mathcal{X}_M$ with respect to the Fobrenius norm according to Lemma \ref{corCvx}. Hence, for any matrices $\textit{\textbf{A}}$ and $\tilde{\textit{\textbf{A}}}$ in $\mathcal{X}_M$, we have:
\begin{multline*}
-\log\big(\det((1-\alpha)\textit{\textbf{A}}+\alpha \tilde{\textit{\textbf{A}}})\big) \leq - (1-\alpha) \log\big(\det(\textit{\textbf{A}})\big) - \alpha \log\big(\det(\tilde{\textit{\textbf{A}}})\big) \\ - \frac{1}{2}\alpha(1-\alpha)\frac{1}{4M^2} \|\tilde{\textit{\textbf{A}}}-\textit{\textbf{A}}\|_F^2.
\end{multline*}
We rearrange terms: 
$$
\log\bigg(\frac{\det\big((1-\alpha)\textit{\textbf{A}}+\alpha \tilde{\textit{\textbf{A}}}\big)}{\det(\textit{\textbf{A}})^{1-\alpha}\det(\tilde{\textit{\textbf{A}}})^{\alpha}}\bigg) \geq \frac{\alpha(1-\alpha)}{8M^2} \|\tilde{\textit{\textbf{A}}}-\textit{\textbf{A}}\|_F^2.
$$
Now, we use the fact that:
$$
 D_{\alpha}\big( \mathcal{N}(0,\textit{\textbf{A}}), \mathcal{N}(0,\tilde{\textit{\textbf{A}}}) \big) = \frac{1}{2(1-\alpha)} \log\bigg(\frac{\det\big((1-\alpha)\textit{\textbf{A}}+\alpha \tilde{\textit{\textbf{A}}}\big)}{\det(\textit{\textbf{A}})^{1-\alpha}\det(\tilde{\textit{\textbf{A}}})^{\alpha}}\bigg)
$$
to get for any matrices $\textit{\textbf{W}} \in \mathbb{R}^{d\times K_1}$ and $\tilde{\textit{\textbf{W}}} \in \mathbb{R}^{d\times K_2}$ such that $\|\tilde{\textit{\textbf{W}}}\tilde{\textit{\textbf{W}}}^T+\sigma^2 \textit{\textbf{I}}_d\|_2\leq M$ and $\|\textit{\textbf{W}}\textit{\textbf{W}}^T+\sigma^2 \textit{\textbf{I}}_d\|_F\leq M$:
$$
 D_{\alpha}\big( P_{\textit{\textbf{W}}},P_{\tilde{\textit{\textbf{W}}}} \big) \geq \frac{\alpha}{16M^2} \|\tilde{\textit{\textbf{W}}}\tilde{\textit{\textbf{W}}}^T+\sigma^2 \textit{\textbf{I}}_d-\textit{\textbf{W}}\textit{\textbf{W}}^T-\sigma^2 \textit{\textbf{I}}_d\|_F^2.
$$
Moreover, for any matrix $\textit{\textbf{W}} \in \mathbb{R}^{d\times K}$ such that the spectral norm of $\textit{\textbf{W}}\textit{\textbf{W}}^T$ is bounded by $B^2$, we have $\|\textit{\textbf{W}}\textit{\textbf{W}}^T+\sigma^2 \textit{\textbf{I}}_d\|_2\leq B^2+\sigma^2$. We conclude using the previous inequality for $M=B^2+\sigma^2$.

\end{proof}

Now, let us go back to the proof of Corollary \ref{corPCAclip}.

\begin{proof}

We assume that there exists a true model $\mathcal{M}_{K_0}$ such that $P^0=P_{\textit{\textbf{W}}_0}$ with $\textit{\textbf{W}}_0 \in \mathbb{R}^{d\times K_0}$ and such that the spectral norm of $\textit{\textbf{W}}_0$ is bounded by $B$ (hence the coefficients of $\textit{\textbf{W}}_0$ are also bounded). As $\textrm{clip}_B$ is a projection onto a closed convex set with respect to the Frobenius norm, we have for any matrix $\textit{\textbf{W}} \in \mathbb{R}^{d\times \hat{K}}$:
$$
\big\| \textrm{clip}_B(\textit{\textbf{W}}\textit{\textbf{W}}^T) - \textrm{clip}_B(\textit{\textbf{W}}_0\textit{\textbf{W}}_0^T) \big\|_F \leq \big\| \textit{\textbf{W}}\textit{\textbf{W}}^T - \textit{\textbf{W}}_0 \textit{\textbf{W}}_0^T \big\|_F
$$
and as the coefficients of $\textit{\textbf{W}}_0\textit{\textbf{W}}_0^T$ are bounded by $B^2$:
$$
\big\| \textrm{clip}_B(\textit{\textbf{W}}\textit{\textbf{W}}^T) - \textit{\textbf{W}}_0 \textit{\textbf{W}}_0^T \big\|_F = \big\| \textrm{clip}_B(\textit{\textbf{W}}\textit{\textbf{W}}^T) - \textrm{clip}_B(\textit{\textbf{W}}_0\textit{\textbf{W}}_0^T) \big\|_F.
$$
According to Lemma \ref{corStrongCvx}, we get for any matrix $\textit{\textbf{W}} \in \mathbb{R}^{d\times \hat{K}}$:
$$
\big\| \textrm{clip}_B(\textit{\textbf{W}}\textit{\textbf{W}}^T) - \textit{\textbf{W}}_0 \textit{\textbf{W}}_0^T \big\|_F^2 \leq \frac{16(B^2+\sigma^2)^2}{\alpha} D_{\alpha}\big( P_{\textit{\textbf{W}}_1},P_{\textit{\textbf{W}}_2} \big) .
$$
Thus:
\begin{multline*}
\mathbb{E} \bigg[ \int \big\| \textrm{clip}_B(\textit{\textbf{W}}\textit{\textbf{W}}^T) - \textit{\textbf{W}}_0 \textit{\textbf{W}}_0^T \big\|_F^2 \tilde{\pi}_{n,\alpha}^{\hat{K}}(dW|X_1^n) \bigg] \\ \leq \frac{16(B^2+\sigma^2)^2}{\alpha} \mathbb{E} \bigg[ \int D_{\alpha}( P_{W}, P_{\textit{\textbf{W}}_0} ) \tilde{\pi}_{n,\alpha}^{\hat{K}}(dW|X_1^n) \bigg]
\end{multline*}
and we use Theorem \ref{corPCA}:
\begin{equation*}
\mathbb{E} \bigg[ \int D_{\alpha}( P_{\textit{\textbf{W}}}, P_{\textit{\textbf{W}}_0} ) \tilde{\pi}_{n,\alpha}^{\hat{K}}(dW|X_1^n) \bigg] = \mathcal{O} \bigg( \frac{dK_0 \log(dn) }{n} \bigg).
\end{equation*}
which ends the proof.

\end{proof}

We can obtain Corollary \ref{corPCAclipfreq} using a simple convexity argument.


\begin{thebibliography}{}

\bibitem[Akaike, 1974]{AIC}
Akaike, H (1974).
\newblock A new look at the statistical model identification.
\newblock {\em IEEE Transactions on Automatic Control}, 19:716--723, 1974.

\bibitem[Alquier and Ridgway, 2017]{Tempered}
Alquier, P. and Ridgway, J. (2017).
\newblock Concentration of tempered posteriors and of their variational
  approximations.
\newblock {\em arXiv preprint arXiv:1706.09293}.

\bibitem[Alquier et~al., 2016]{alquier2016properties}
Alquier, P., Ridgway, J., and Chopin, N. (2016).
\newblock On the properties of variational approximations of {G}ibbs
  posteriors.
\newblock {\em JMLR}, 17(239):1--41.

\bibitem[Behrens et~al., 2012]{behrens2012tuning}
Behrens, G., Friel, N., and Hurn, M. (2012).
\newblock Tuning tempered transitions.
\newblock {\em Statistics and computing}, 22(1):65--78, 2012.

\bibitem[Bhattacharya et~al., 2016]{bhattacharya2016bayesian}
Bhattacharya, A., Pati, D., and Yang, Y. (2016).
\newblock Bayesian fractional posteriors.
\newblock {\em arXiv preprint arXiv:1611.01125, to appear in the Annals of
  Statistics}.

\bibitem[Bhattacharya et~al., 2018]{Plage}
Bhattacharya, A., Pati, D., and Yang, Y. (2018).
\newblock On statistical optimality of variational {Bayes}.
\newblock {\em Proceedings of Machine Learning Research}, 84 - AISTAT.

\bibitem[Bishop, 1999]{VariationalComponents}
Bishop, C. (1999).
\newblock Variational Principal Components.
\newblock {Proceedings Ninth International Conference on Artificial Neural Networks, ICANN'99}, 1999.

\bibitem[Blei et~al., 2017]{blei2017variational}
Blei, D.M., Kucukelbir, A., and McAuliffe, J.D. (2017).
\newblock Variational inference: A review for statisticians.
\newblock {\em Journal of the American Statistical Association}, 2017.

\bibitem[Cai et~al., 2017]{MatrixEstimationWu}
Cai, T. and Ma, Z. and Wu, Y. (2015).
\newblock Optimal estimation and rank detection for sparse spiked covariance matrices.
\newblock {\em Probability Theory and Related Fields}, 2015.

\bibitem[Catoni, 2007]{MR2483528}
Catoni, O. (2007).
\newblock {\em {PAC}-{B}ayesian supervised classification: the thermodynamics
  of statistical learning}.
\newblock Institute of Mathematical Statistics Lecture Notes---Monograph
  Series, 56. Institute of Mathematical Statistics, Beachwood, OH.

\bibitem[Ch{\'e}rief-Abdellatif and Alquier, 2018]{cherief2018consistency}
Ch{\'e}rief-Abdellatif, B.-E. and Alquier, P. (2018).
\newblock Consistency of variational bayes inference for estimation and model
  selection in mixtures.
\newblock {\em Electron. J. Statist.}, 12(2):2995--3035.

\bibitem[Forth et~al., 2014]{MatrixBook}
Forth, S. and Hovland, P. and Phipps, E. and Utke, J. and Walther, A. (2014).
\newblock Recent Advances in Algorithmic Differentiation.
\newblock {Springer Publishing Company, Incorporated}, 2014.

\bibitem[Ghosal et~al., 2000]{ghosal2000convergence}
Ghosal, S., Ghosh, J.K., and Van~der Vaart, A. (2017).
\newblock Convergence rates of posterior distributions.
\newblock {\em Annals of Statistics}, pages 500--531, 2000.

\bibitem[Gr{\"u}nwald and Van~Ommen, 2017]{grunwaldmisspecifiation}
Gr{\"u}nwald, P.~D. and Van~Ommen, T. (2017).
\newblock Inconsistency of {B}ayesian inference for misspecified linear models,
  and a proposal for repairing it.
\newblock {\em Bayesian Analysis}, 12(4):1069--1103.

\bibitem[Lv and Liu, 2013]{modelselectionmisspecification}
Lv, J. and Liu, J. S. (2013).
\newblock Model selection principles in misspecified models.
\newblock {\em Journal of the Royal Statistical Society: Series B (Statistical Methodology)}, 2013.

\bibitem[Massart, 2005]{MR2319879}
Massart, P. (2007).
\newblock {\em Concentration inequalities and model selection}.
\newblock Saint-Flour Summer School on Probability Theory 2003 (Jean Picard
  ed.), Lecture Notes in Mathematics. Springer, 2007.
 
\bibitem[Moridomi et~al., 2018]{OnlineLogDet}
Moridomi, K.I. and Hatano, K. and Takimoto, E. (2018).
\newblock {Online linear optimization with the log-determinant regularizer}.
\newblock {\em IEICE Transactions on Information and Systems}, 2018.
  
\bibitem[Rao and Wu, 2001]{ModelSelectionReview}
Rao, C. R. and Wu, Y. (2001).
\newblock On model selection.
\newblock Lecture Notes--Monograph Series, Institute of Mathematical Statistics. 2001.

\bibitem[Schwarz, 1978]{BIC}
Schwarz, G. (1978).
\newblock Estimating the dimension of a model.
\newblock {\em The Annals of Statistics}, 6(2):461--464, 1978.

\bibitem[Van~Erven and Harremos, 2014]{van2014renyi}
Van~Erven, T., and Harremos, P. (2014).
\newblock {R{\'e}nyi divergence and Kullback-Leibler divergence}.
\newblock {\em IEEE Transactions on Information Theory}, 60(7):3797--3820, 2014.
  
\bibitem[Wang and Blei, 2018]{wang2018frequentist}
Wang, Y. and Blei, D.~M. (2018).
\newblock Frequentist consistency of variational {B}ayes.
\newblock Journal of the American Statistical Association (to appear).

\bibitem[Vehtari et~al., 2018]{CVVehtari}
Vehtari, A. and Tolvanen, V. and Mononen, T. and Winther, O. (2014).
\newblock Bayesian Leave-One-Out Cross Validation Approximations for Gaussian Latent Variable Models.
\newblock Journal of Machine Learning Research. 2014

\bibitem[Yang, 2005]{yang2005can}
Yang, Y. (2005).
\newblock Can the strengths of {AIC} and {BIC} be shared? {A} conflict between
  model identification and regression estimation.
\newblock {\em Biometrika}, 92(4):937--950, 2005.

\bibitem[Zhang and Gao, 2017]{Chicago}
Zhang, F. and Gao, C. (2017).
\newblock Convergence rates of variational posterior distributions.
\newblock {\em arXiv preprint arXiv:1712.02519v1}.

\end{thebibliography}
\end{document}